\documentclass[11pt,english]{smfart}
\usepackage[english,francais]{babel}
\usepackage{amsmath}
\usepackage{amssymb}
\usepackage{amsthm}
\usepackage{mathrsfs}

\usepackage{url}
\usepackage{hyperref}

\usepackage[dvips]{graphicx}

\def\Dio{{\rm dio}}

\newtheorem{theo}{Theorem}

\newtheorem{prop}[theo]{Proposition}

\newtheorem{coro}[theo]{Corollary}

\theoremstyle{definition}

\theoremstyle{remark} 
\newtheorem{rem}[theo]{Remark}
\theoremstyle{definition}

\newcommand{\cal}{\mathcal}

\title{On the expansion of some exponential periods in an integer base}\author[B.~Adamczewski]{Boris Adamczewski}

\address{CNRS, Universit\'e de Lyon, Universit\'e Lyon 1 \\ 
Institut Camille Jordan \\ 
43 boulevard du 11 novembre 1918  \\
69622 Villeurbanne cedex France}

\email{Boris.Adamczewski@math.univ-lyon1.fr}
\urladdr{http://math.univ-lyon1.fr/~adamczew/}

\begin{document}

\begin{abstract}
We derive a lower bound for the subword complexity of the base-$b$ expansion ($b\geq 2$) 
of all real numbers whose irrationality exponent is equal to $2$.  
This provides a generalization of a 
theorem due to Ferenczi and Mauduit. 
As a consequence, we obtain the first lower bound for 
the subword complexity of  the number $e$ and 
of some other transcendental exponential periods.  
\end{abstract}

\maketitle

\section{Introduction}

The decimal expansion of real numbers like $\sqrt 2$, $\pi$ or $e$ 
appears to be quite mysterious and, for a long time, has baffled  
mathematicians. While  numerical observations seem to speak in favour of a  
complex structure, most questions one may imagine to ask about the decimal expansion 
of classical irrational constants turn out to be out of reach.  

Kontsevitch and Zagier \cite{KZ} offered a promising framework to try to 
distinguish usual constants from other real numbers by introducing the notions of {\it period} and 
of {\it exponential period}. Algebraic numbers, $\pi$, $\log 2$ and $\zeta(3)$ 
are periods, while $e$ is conjecturally not a period. 
However, $e$ is a typical example of an exponential period. 
Exponential periods form a countable set that contains the set of periods. 
We refer the reader to \cite{KZ} for exact definitions and more results 
about both notions, but to paraphrase these authors, all classical contants 
are periods in an appropriate sense. Folklore suggests that all irrational 
periods are {\it normal numbers}.   
Recall that a real number is a normal number if 
for every integer $b\geq 2$ and every positive integer $n$,
each one of the $b^n$ blocks of length $n$
over the alphabet $\{0, 1, \ldots, b-1\}$
occurs in its base $b$ expansion with frequency $1/b^n$. 
This notion was introduced in 1909 by Borel \cite{Bo} who 
proved that almost all numbers, with respect to the Lebesgue measure,   
are normal despite the uncomfortable fact  
that not a single natural example of a normal number is known. 

\medskip

An interesting (and perhaps more reasonable) way to tackle problems concerned with 
the expansions of classical constants in integer bases  
is to consider the {\it subword complexity} of real numbers. 
Let $\xi$ be a real number and $b\geq 2$ be a positive integer. Then $\xi$ has a unique expansion 
in the base $b$, that is, there exists a unique sequence ${\bf a}=(a_n)_{n\geq -k}$ with values 
in $\{0,1,\ldots,b-1\}$ such that 
$$\begin{array}{rl}
\xi=& \displaystyle\sum_{n\geq -k} \frac{a_n}{b^n} \\ 
:= &a_{-k}a_{-k+1}\cdots a_{-1}a_0{\scriptscriptstyle \bullet} a_1a_2\cdots
\end{array}$$
The complexity function of $\xi$ with respect to the base $b$ is the function that associates  
with each positive integer $n$ the positive integer  
$$
p(\xi,b,n) := \mbox{Card} \{(a_j,a_{j+1},\ldots, a_{j+n-1}), \; j \geq 1\}.
$$
A normal number thus has the maximum possible complexity in every integer base, 
that is, $p(\xi,b,n) = b^n$ for every 
positive integer $n$ and every integer $b\geq 2$. As mentioned before, one usually  expects 
such a high complexity for numbers like $\sqrt 2$, $\pi$ and $e$.  
This problem was first addressed in 1938 by Hedlund and Morse \cite{HM}.

\medskip

In their paper, Hedlund and Morse obtained a fundamental result that can be restated as follows.

\medskip

\noindent
{\bf\itshape Theorem HM.} --- 
{\it Let $b\geq 2$ be an integer and $\xi$ be a real number. 
Then $\xi$ is rational if and only if it has a bounded complexity function. 
Furthermore, if $\xi$ is irrational, 
its complexity function is increasing and thus 
$$
p(\xi,b,n) - n \geq 1, \;\; \forall n\geq 1.
$$
}

\medskip

To find lower bounds for the complexity function of algebraic irrational numbers is a challenging 
problem. In 1997, Ferenczi and Mauduit \cite{FM} proved the theorem below.  
Actually, their result is slightly weaker and 
the present statement is given according to a clever remark of Allouche 
outlined in \cite{Al}. 

\medskip

\noindent
{\bf\itshape Theorem FM.} --- 
{\it Let $b \geq 2$ be an integer and $\xi$ be an algebraic irrational number.  
Then, 
\begin{equation}\label{bfm}
\lim_{n\to \infty}  p(\xi,b,n) - n = + \infty.
\end{equation}}

\medskip

The proof of Theorem FM mixes techniques from combinatorics on words 
and Diophantine approximation. The main ingredient is a $p$-adic version 
of Roth's theorem due to Ridout \cite{Rid57}.  
Recently, Bugeaud and the author \cite{AB1} (see also \cite{ABL}) improved Theorem FM 
by means of the Schmidt Subspace Theorem. Under the same assumption, 
these authors proved that (\ref{bfm}) can be replaced by 
$$
\lim_{n\to \infty} \frac{p(\xi,b,n)}{n} = + \infty.
$$

\medskip

The situation regarding transcendental constants is even worse: 
apparently, there is not a single trancendental exponential period for which 
one knows an improvement of the lower bound given by Theorem HM. Of course, 
one could choose an exponential period $\xi$ and compute the first 
digits of its base-$b$ expansion. If one is lucky enough, one will find occurrences of 
many different blocks of digits of a given length. 
But, this would only lead, for some integer $k$, to a lower bound of the type 
$$
p(\xi,b,n) - n \geq k,   
$$ 
for all sufficiently large positive integers $n$.

\medskip

 Recall that the {\it irrationality exponent} of an irrational number $\xi$, denoted by $\mu(\xi)$, 
 is defined as the supremum of the real numbers $\rho$ for which the inequality 
$$
\left\vert \xi - \frac{p}{q} \right\vert < \frac{1}{q^{\rho}}
$$
has infinitely many different rational solutions $p/q$. 
It always satifies 
$$
2 \leq \mu(\xi) \leq +\infty.
$$
The set of real numbers whose irrationality exponent is equal to $2$ has full Lebesgue measure.  
By Roth's theorem \cite{Roth}, algebraic irrational numbers all have an irrationality exponent 
equal to $2$.   
The aim of this note is to generalize Theorem FM as follows.

\begin{theo}\label{main}
Let $b\geq 2$ be an integer and $\xi$ be an irrational real number such that 
$\mu(\xi)=2$. Then,
\begin{equation}\label{complex}
\lim_{n\to \infty} p(\xi,b,n) - n = + \infty.
\end{equation}
\end{theo}

Our proof of Theorem \ref{main} is essentially a combination of known results 
that rely on fine combinatorial properties of infinite words with a very low complexity. 
In particular, we use, in an essential way, a result due to Berth\'e, Holton and Zamboni 
\cite{BHZ} concerning initial repetitions occurring in Sturmian words.

\medskip

We derive the following consequences of Theorem \ref{main}. 

\begin{coro}\label{e}
For every integer $b\geq 2$, the number $e$ satisfies (\ref{complex}).
\end{coro}

To our knowledge, this is the first example of a transcendental exponential period for which 
we can improve the bound of Theorem HM. 

The only property of the number $e$ used in the proof of Corollary \ref{e} is 
that $\mu(e)=2$, which follows from Euler's formula for the continued fraction expansion of 
$e$ (see the proof of Corollary \ref{e} in Section \ref{proofe}). 
Actually, many other examples 
of numbers involving the exponential function, trigonometric functions, or the modified 
Bessel function at rational arguments also have an irrationality exponent equal to $2$.  
In particular, the same conclusion holds in Corollary \ref{e}, if we replace the number $e$ 
by any of the following numbers (see \cite{Bundschuh,Tasoev}):  
$$\begin{array}{l}
e^{a},  \; a\in \mathbb Q, a\not= 0 ;\\ \\
\displaystyle \tan\left(\frac{1}{a}\right), \sqrt a \tan\left(\frac{1}{\sqrt a}\right), \frac{1}{\sqrt a} 
\tan\left(\frac{1}{\sqrt a}\right), \; a \in \mathbb N, a\not= 0 ; \\ \\
\displaystyle\tanh \left(\frac{2}{a}\right), \; a \in \mathbb N, a\not= 0 ;\\ \\
\displaystyle\sqrt{\frac{v}{u}}\tanh \left(\frac{1}{\sqrt{uv}}\right), \; u,v \in \mathbb N, uv\not= 0.
\end{array}
$$
Other interesting values covered by our approach are the numbers 
$$
\frac{ J_{(p/q)+1}(2/q)}{J_{p/q}(2/q)} \; \mbox{ and } \; \frac{ I_{(p/q)+1}(2/q)}{I_{p/q}(2/q)}, 
\; p/q \in \mathbb Q, 
 $$
where $\displaystyle J_{\lambda}(z)=\left(\frac{z}{2}\right)^{\lambda}\sum_{n= 0}^{+\infty}
\frac{(iz/2)^{2n}}{n! \; \Gamma(\lambda+n+1)}$ denotes the Bessel function 
of the first kind and $\displaystyle I_{\lambda}(z)=\sum_{n= 0}^{+\infty}
\frac{(z/2)^{\lambda + 2n}}{n! \; \Gamma(\lambda+n+1)}$ 
denotes the modified 
(or hyperbolic) Bessel function of the first kind (see for instance \cite{Komatsu}).   

Furthermore, multiplying any of these numbers by a nonzero rational and then 
adding a rational leads to a new example for which our result can be applied. We can also take the natural action of $\mbox{GL}_2(\mathbb Z)$ on any of these numbers. That is, starting from one of the 
above number $\xi$, the number $(a\xi+b)/(c\xi+d)$, where $\vert ad-bc\vert=1$, also satisfies 
the bound (\ref{complex}) of Theorem \ref{main}.

\medskip

Another consequence of our Theorem \ref{main} is that a real number with 
a bounded sequence of partial quotients in its continued fraction expansion 
cannot have base-$b$ expansions that are too simple. 

\begin{coro}\label{bad}
Let $b\geq 2$ be an integer and $\xi$ be an irrational  real number whose continued fraction 
expansion is $[a_0,a_1,a_2,\ldots]$. If the sequence of integers 
$(a_n)_{n\geq 0}$ is bounded, then 
$\xi$ satisfies (\ref{complex}).
\end{coro}

\section{Repetitions in infinite words with a low complexity}

We first introduce some notation from combinatorics on words.

\medskip

Let ${\cal A}$ be a finite set. The length of a word
$W$ on the alphabet ${\cal A}$, that is, the number of letters
composing $W$, is denoted by $\vert W\vert$. If $a$ is a letter and $W$ a finite word,  
then $\vert W\vert_a$ denotes the number of occurrences of the letter $a$ in $W$. 
For any positive integer $k$, we write
$W^k$ for the word 
$$\underbrace{W\cdots W}_{\mbox{$k$ times}}$$ 
(the concatenation
of the word $W$ repeated $k$ times). 
More generally, for any positive real number
$x$,  $W^x$ denotes the word
$W^{\lfloor x \rfloor}W'$, where $W'$ is the prefix of
$W$ of length $\left\lceil(x-\lfloor x\rfloor)\vert W\vert\right\rceil$. 
Here,  $\lfloor x \rfloor$ and 
$\lceil x\rceil$ denote the floor and ceiling functions, respectively. 

\medskip

We now consider two exponents that measure repetitions occurring in infinite words. 
They were introduced in \cite{BHZ} and \cite{AB3} (see also \cite{AC}), respectively. 
The first exponent,  the {\it initial critical exponent} 
of an infinite word ${\bf a}$, is defined as the supremum of 
all positive real numbers $\rho$ for which there exist arbitrarily long prefixes $V$ such that 
$V^{\rho}$ is also a prefix of ${\bf a}$. 
The second exponent, the {\it Diophantine exponent} 
of an infinite word ${\bf a}$, is defined as 
the supremum of the real numbers $\rho$ for which there exist arbitrarily long prefixes of ${\bf a}$ 
that can be factorized as $UV^w$, where $U$ and $V$ are two finite words 
($U$ possibly empty) and $w$ is a real number such that  
$$
\frac{\vert UV^w\vert}{\vert UV \vert} \geq \rho.
$$

\medskip

The initial critical exponent and the Diophantine exponent of ${\bf a}$ 
are respectively denoted by 
$\mbox{ice}({\bf a})$ and 
$\Dio({\bf a})$. Both exponents are clearly related by the following relation
\begin{equation}\label{icedio}
1 \leq \mbox{ice}({\bf a}) \leq \Dio({\bf a}) \leq +\infty.
\end{equation}

\medskip

Recall that the subword complexity function of an infinite word ${\bf a}=a_1a_2\cdots$ 
is the function   that associates  
with each positive integer $n$ the positive integer  
$$
p({\bf a},n) := \mbox{Card} \{(a_j,a_{j+1},\ldots ,a_{j+n-1}), \; j \geq 1\}.
$$

\medskip

We now prove the following result.

\begin{prop}\label{prop}
Let ${\bf a}$ be an infinite word such that the difference $p({\bf a},n) - n$ is bounded. Then  
$$
\Dio({\bf a}) > 2.
$$
\end{prop}

\begin{proof}
Our first step is a reduction argument that was previously outlined by Allouche in a similar 
context \cite{Al}. More precisely, we are going to prove that it is sufficient  
for our purpose to focus on the Diophantine exponent of Sturmian words.  
Sturmian words are defined as the binary words ${\bf s}$ for which   
$p({\bf s},n)=n+1$ for every positive integer $n$. 
Recall that the set of Sturmian words is uncountable. 

From now on, we fix an infinite word ${\bf a} = a_1a_2\cdots$ defined over a finite alphabet 
${\cal A}$ and such that the difference $p({\bf a},n) -n$ is bounded. 
If ${\bf a}$ is eventually periodic, it is easily checked that $\Dio({\bf a})$ is infinite. 
We can thus assume that ${\bf a}$ is aperiodic.  
Theorem HM thus implies that the subword complexity function of 
${\bf a}$ is increasing.  
Consequently, the difference $p({\bf a},n) -n$ is a nondecreasing sequence of bounded integers. 
Such sequence is eventually constant and there 
thus exists two positive integers $k$ and $n_0$ such that 
\begin{equation}\label{quasisturm}
p({\bf a},n) = n + k, \;\; \forall n\geq n_0.
\end{equation}
By a result of Cassaigne \cite{Cas}, an infinite word ${\bf a}$ satisfies Equality (\ref{quasisturm}) 
if and only if there exist a finite word $W$, a Sturmian infinite word ${\bf s}$ defined over $\{0,1\}^*$ 
and a nonerasing morphism $\varphi$ from the free monoid $\{0,1\}^*$ into ${\cal A}^*$ such that 
\begin{equation}\label{decomp}
{\bf a} = W \varphi({\bf s}).
\end{equation}
Our infinite word ${\bf a}$ thus has such a decomposition and we claim that  
\begin{equation}\label{sturm}
\Dio({\bf a}) \geq \Dio({\bf s}).
\end{equation}
Set ${\bf s} = s_1s_2\cdots$. 
To prove that (\ref{sturm}) holds, we only need a classical property of Sturmian words: 
each of the two letters occurring in a Sturmian word has a frequency. 
This means that there exists a real number $\alpha$ in $(0,1)$ 
such that 
$$
\lim_{n \to \infty} \frac{\vert s_1s_2\cdots s_n\vert_1}{n} = \alpha,
$$  
and consequently, 
$$
\lim_{n \to \infty} \frac{\vert s_1s_2\cdots s_n\vert_0}{n} = 1 - \alpha.
$$
The number $\alpha$ is always irrational and is termed the slope of ${\bf s}$. 
It follows that  
\begin{equation}\label{o}
\vert \varphi(s_1s_2\cdots s_n) \vert = \delta n + o(n),
\end{equation}
where $\delta := (\alpha \vert \varphi(a)\vert + (1-\alpha) \vert \varphi(b) \vert)$. 
Here and in the sequel, $o$ stands for the usual Landau notation.

\medskip

Now, let us assume that $\Dio({\bf s}) = \rho$ for some positive real number $\rho$, 
and let $\varepsilon$ be a positive real 
number. By definition of the Diophantine 
exponent, there exists an infinite sequence of prefixes of ${\bf s}$ that can be factorized as 
$U_nV_n^{w_n}$, with $\vert U_nV_n^{w_n} \vert / \vert U_nV_n \vert > \rho -\varepsilon$ and 
such that the sequence $(\vert U_nV_n \vert)_{n\geq 1}$ increases. 
We then infer from (\ref{decomp}) that ${\bf a}$ begins with the word 
$$
W\varphi(U_nV_n^{w_n}).
$$
Set $A_n = W \varphi(U_n)$ and $B_n = \varphi(V_n)$. 
There thus exists a positive real number $r_n$ such that 
$$
W\varphi(U_nV_n^{w_n}) = A_nB_n^{r_n}.
$$
Since $\vert U_nV_n \vert$ can be chosen arbitrarily large, we infer from Equality (\ref{o}) that 
$$
\vert W \varphi(U_nV_n^{w_n})\vert = \delta \vert U_nV_n^{w_n} \vert + o(\vert U_nV_n^{w_n} \vert)
$$
and 
$$
\vert \varphi(U_nV_n) \vert = \delta \vert U_nV_n \vert + o(\vert U_nV_n\vert).
$$
Consequently, 
$$
\vert A_nB_n^{r_n} \vert / \vert A_nB_n \vert \geq \rho - 2\varepsilon,
$$
for every $n$ large enough. This shows that 
\begin{equation}\label{dio}
\Dio({\bf a}) \geq \rho = \Dio({\bf s}) 
\end{equation}
as desired.

\medskip

We have now to distinguish two cases depending on 
the Diophantine properties of the slope $\alpha$ of the Sturmian words ${\bf s}$.  
Let us denote by $[0,m_1,m_2,\ldots]$ the continued fraction expansion of $\alpha$. 

First,  let us assume that $\alpha$ has a bounded sequence of partial quotients, 
say bounded by a positive integer $M$. Then there are only a finite number of distinct  
pairs $(a_j,a_{j+1})$ and {\it a fortiori} of distinct triples $(a_j,a_{j+1},a_{j+2})$. 
By the pigeonhole principle, there exist either a pair of integers 
$(s,t)$, $2\leq s\leq M$, $1\leq t \leq M$, such that $(a_j,a_{j+1})=(s,t)$ for infinitely many indices $j$, 
or there exist infinitely many indices $j$ such that 
$(a_j,a_{j+1},a_{j+2}) = (1,1,1)$.  
In all cases, mixing Propositions 5.1 and 5.2 of \cite{BHZ} we get that  
$$
\mbox{ice}({\bf s}) \geq 2  + \frac{1}{2(M+1)^2+1} >2.
$$
We thus infer from Inequality (\ref{icedio}) that 
\begin{equation}\label{dio2}
\Dio({\bf s}) >2.
\end{equation}

On the other hand, if $\alpha$ has an unbounded 
sequence of partial quotients,  it is shown in Proposition 11.1 of \cite{AB2} that 
\begin{equation}\label{dio1}
\Dio({\bf s}) = +\infty.
\end{equation}

To sum up, (\ref{dio}), (\ref{dio2}) and (\ref{dio1}) give 
$$
\Dio({\bf a}) > 2,
$$
concluding the proof. 
\end{proof}

\section{Diophantine exponent and rational approximations}\label{proofe}

We now briefly recall some interplay between the Diophantine exponent and 
the irrationality exponent that can be found in \cite{AB2}.  
Let $\xi$ be a real number whose base-$b$ expansion is 
$0{\scriptscriptstyle \bullet} a_1a_2\cdots$. Set ${\bf a}=a_1a_2\cdots$. 
Then the Diophantine exponent of ${\bf a}$ and the irrationality exponent of $\xi$ are 
linked by the following inequality:
\begin{equation}\label{diorat}
\Dio({\bf a}) \leq \mu(\xi).
\end{equation}
Indeed, let us assume that the word ${\bf a}$ begins with a prefix of the form $UV^w$.
Set $q = b^{\vert U\vert}(b^{\vert V\vert}-1)$. 
A simple computation shows that there exists an integer $p$ such that 
$$
p/q = 0{\scriptscriptstyle \bullet} UVVV\cdots.
$$  
Since $\xi$ and $p/q$ have the same first $\vert UV^w \vert$ digits, 
we obtain that 
$$
\left\vert \xi - \frac{p}{q} \right\vert < \frac{1}{b^{\vert UV^w\vert}}
$$
and thus 
\begin{equation}\label{min}
\left\vert \xi - \frac{p}{q} \right\vert < \frac{1}{q^{\rho}},
\end{equation}
where $\rho = \vert UV^w\vert/\vert UV\vert$. 
We do not claim here that $p/q$ is written in lowest terms.  Actually, it may well happen that the 
$\gcd$ of $p$ and $q$ is quite large but (\ref{min}) still holds in that case.
Inequality (\ref{diorat}) then follows directly from the definition of both exponents.

\medskip

We are now ready to conclude the proof of our main results.

\begin{proof}[Proof of Theorem \ref{main}]  
It is a straightforward consequence of Proposition \ref{prop} and Inequality (\ref{diorat}). 
\end{proof}

\begin{proof}[Proof of Corollary \ref{e}]  
It is known after Euler\footnote{Actually, this formula seems to have been discovered first by R. Cotes; 
Euler would be the first to give a proof.} 
that the 
continued fraction expansion of $e$ has very special patterns; namely 
\begin{equation}\label{euler}
e = [2,1,2,1,1,4,1,1,6,1,1,8,\ldots,1,1,2n,1,1\ldots].
\end{equation}
From Euler's formula, we can easily derive that  
the irrationality exponent of $e$ satisfies $\mu(e) = 2$, concluding the proof. 
Indeed, if $q_n$ and $a_n$ respectively denote the $n$-th convergent and the $n$-th partial 
quotient of a real number $\xi$, the theory of continued fractions ensures that
$$
\mu(\xi) = \limsup_{n\to\infty} \left(2 + \frac{\log a_{n+1}}{\log q_n}\right).
$$ 
\end{proof}

\begin{proof}[Proof of Corollary \ref{bad}]  
 A basic result from the theory of continued fractions ensures that 
 $\mu(\xi)=2$ when $\xi$ has a bounded sequence of partial quotients in its continued fraction 
 expansion (see for instance \cite{Khintchine}, Theorem 23, Page 36). 
\end{proof}

\section{Comments}\label{last}

We end this note with few comments.

\begin{rem}\label{FeMa} As already mentioned, the main ingredient 
in the proof of Theorem FM is a $p$-adic version of Roth's theorem due to Ridout. 
Incidentally, our proof of Theorem \ref{main} 
provides a new proof of Theorem FM. 
In particular, this shows that it   can be obtained with the use of 
Roth's theorem only. That is, $p$-adic considerations are unnecessary to prove Theorem FM.  

However, the use of some $p$-adic information by Ferenczi and Mauduit in \cite{FM} turned out to be 
of great importance since it led to the main lower bound for the complexity of 
algebraic irrational numbers obtained in \cite{AB1}.
\end{rem}

\begin{rem}
It is likely that Proposition \ref{prop} could be improved to  
\begin{equation}\label{dioconj}
\Dio({\bf a}) \geq \frac{3+\sqrt 5}{2}.
\end{equation} 
This result would be optimal since the bound is reached for the Fibonacci word, which is 
the most famous example among Sturmian words. We note that, as a consequence of 
irrationality measures obtained by Baker in \cite{Ba}, Inequality (\ref{dioconj})  
would permit us to show that the conclusion of Corollary \ref{e} still holds if we replace 
the number $e$ by $\log (1 + 1/n)$, for every integer $n\geq 68$. 
However, our approach would not permit us to deduce  
a new lower bound for the complexity of expansions in integer bases of 
periods like $\pi$, $\log 2$ or 
$\zeta(3)$ from (\ref{dioconj}) since the best known upper bounds for the irrationality exponent of 
these numbers are all larger than $(3+\sqrt 5)/2$.
\end{rem}

\begin{rem} 
As a limitation of our approach, we quote that there are  
real numbers with a low complexity in some integer base but with 
an irrationality exponent equal to $2$. 
For instance, it was proved in  \cite{Sh} that the binary number 
$$
\xi := \sum_{n\geq 0}\frac{1}{2^{2^n}}
$$
has bounded partial quotients in its continued fraction expansion, 
while a classical theorem of Cobham implies that $p(\xi,2,n) =O(n)$ 
(see for instance Corollary 10.3.2 of \cite{AS}, Page 304). 
Actually, one can deduce from the proof of Lemma 2.4 in \cite{Gh} 
the more precise upper bound $p(\xi,2,n) \leq (2+\ln 3)n + 4$, 
for every positive integer $n$. 
\end{rem}

\begin{rem}
It would also be very interesting to investigate the complexity of expansions of $e$ 
and other contants from a computational point of view. In this direction, we pose 
the following  open question. Note that similar open questions were posed by Allouche and Shallit 
\cite{AS}, Page 402, Problems 3 and 4.

\medskip

\noindent
{\bf \itshape Problem.}\,\,--- {\it Prove that the decimal expansion of $e$ cannot be produced by a 
finite automaton. }
\end{rem}

\bigskip

\noindent{\bf \itshape Acknowledgements.}\,\,--- 
The author thanks the ANR for its support through the project ``DyCoNum"--JCJC06 134288. 
He is grateful to Yann Bugeaud 
and Michel Waldschmidt for their comments on an earlier draft of this note.  
The author is also indebted to Tanguy Rivoal for suggesting relevant references and to 
the anonymous referee for his very careful reading that helped 
him to improve the presentation of the present work. 

\end{document}